\documentclass{amsart}

\usepackage{amssymb}
\usepackage{amsthm}
\usepackage{amsmath}
\usepackage{enumerate}
\usepackage{mathrsfs}

\usepackage{varioref}
\usepackage{float}

 \theoremstyle{plain}

\newtheorem{thm}{Theorem}[section]
\newtheorem{cor}[thm]{Corollary}
\newtheorem{lem}[thm]{Lemma}

\begin{document}

\title{On the number of conjugacy classes of a permutation group}

\author{Martino Garonzi} \address{Dipartimento di Matematica,
Universit\`a degli studi di Padova, Via Trieste 63, 35100
Padova, Italy} \email{mgaronzi@math.unipd.it}

\author{Attila Mar\'oti} \address{Fachbereich
Mathematik, Universit\"{a}t Kaiserslautern, Postfach 3049, 67653
Kaiserslautern, Germany} \email{maroti@mathematik.uni-kl.de}

\thanks{The first author wishes to thank the hospitality and financial support of Technische Universit{\"a}t Kaiserslautern where much of this work was carried out. The research of the second author was supported
by an Alexander von Humboldt Fellowship for Experienced
Researchers and by OTKA K84233.}
\date{\today}

\begin{abstract}
We prove that any permutation group of degree $n \geq 4$ has at most $5^{(n-1)/3}$ conjugacy classes.
\end{abstract}

\maketitle

\section{Introduction}

One of the most important invariants of a finite group $G$ is the number $k(G)$ of its conjugacy classes. This is equal to the number of complex irreducible characters of $G$. There are many interesting open problems concerning $k(G)$. For example it is not known whether there exists a universal constant $c >0$ so that $k(G) > c \log (|G|)$ holds for any finite group $G$ (see \cite{P}, \cite{K}, \cite{J}). In this paper we are interested in giving upper bounds for $k(G)$. Such problems are closely related to the $k(GV)$ theorem (see \cite{schmidbook}) and the non-coprime $k(GV)$ problem \cite{GT}.

One of the important special cases in giving upper bounds for $k(G)$ is the case when $G$ is a permutation group of degree $n$. Kov\'acs and Robinson \cite{KR} proved that $k(G) \leq 5^{n-1}$ and reduced the proposed bound of $k(G) \leq 2^{n-1}$ to the case when $G$ is an almost simple group. This latter bound was later proved by Liebeck and Pyber in \cite{LP} for arbitrary finite groups $G$. Kov\'acs and Robinson in \cite{KR} also proved that $k(G) \leq 3^{(n-1)/2}$ for $G$ a solvable permutation group of degree $n \geq 3$. Later Riese and Schmid \cite{RS} proved the same bound for $3'$, $5'$ and $7'$-groups, and in \cite{Ma2} the second author obtained the bound $k(G) \leq 3^{(n-1)/2}$ for an arbitrary finite permutation group $G$ of degree $n \geq 3$. 

By imposing restrictions on the set of composition factors of the permutation group $G$, one can obtain stronger bounds on $k(G)$. For example, in \cite{Ma2} it was shown that $k(G) \leq {(5/3)}^{n}$ whenever $G$ has no composition factor isomorphic to $C_{2}$, and more recently Schmid \cite{schmid} proved that $k(G) \leq 7^{(n-1)/4}$ for $n \geq 5$ where $G$ has no non-abelian composition factor isomorphic to an alternating group or a group in \cite{ATLAS}. However it seems hard to generalize these bounds for arbitrary groups. 

The main result of the current paper is the following.

\begin{thm}
\label{main}
A permutation group of degree $n \geq 4$ has at most $5^{(n-1)/3}$ conjugacy classes.
\end{thm}

The direct product of $n/4$ copies of $S_4$ or $D_8$ is a permutation group of degree $n$ with exactly $5^{n/4}$ conjugacy classes (whenever $n$ is a multiple of $4$). But even more can be said. Pyber has pointed out (see \cite{KR} and also \cite{LP}) that for each constant $0< c < 5^{1/4}$ there are infinitely many transitive permutation groups $G$ with $k(G) > c^{n-1}$. In fact, $G$ can be taken to be the transitive $2$-group $D_{8} \wr C_{n/4} \leq S_n$ whenever $n$ is a power of $2$ at least $4$. (This can be seen by (1) of Lemma \ref{ineq}.)   

However, for special subgroups of primitive permutation groups $G$, one may give better than exponential bounds for $k(G)$. A transitive permutation group $G$ is called primitive if the stabilizer of any point is a maximal subgroup in $G$. This is equivalent to saying that the only blocks of imprimitivity for $G$ are the singleton sets and the whole set on which $G$ acts. The symmetric group $S_{n}$ is always primitive and it is easy to see that $k(S_{n}) = p(n)$, the number of partitions of $n$. Hardy and Ramanujan \cite{HR} and independently but later Uspensky \cite{U} gave an asymptotic formula for $p(n)$ and this is less than exponential. It is a natural question whether $k(G) \leq p(n)$ for any primitive permutation group of degree $n$. This was shown to be true for sufficiently large $n$ by Liebeck and Pyber \cite{LP} and later for all normal subgroups of all primitive groups by the second author in \cite{Ma2}. In this paper we go even further by showing that for any subgroup $H$ of any primitive permutation group $G$ of degree $n$, apart from the alternating group $A_n$ and $S_n$, we have $k(H) \leq p(n)$ (see Theorem \ref{subprim}). This result is used to give a general upper bound for $k(G)$ for a transitive permutation group $G$ from knowledge of the partition function (see Theorem \ref{part}). Finally, this result is used to derive Theorem \ref{main}.   


\section{Preliminaries}

The following lemma collects basic information on the number of conjugacy classes in a subgroup and in a normal subgroup of a finite group. 

\begin{lem}
\label{ineq}
Let $H$ be a subgroup and $N$ be a normal subgroup of a finite group $G$. Then 
\begin{enumerate}
\item $k(H)/|G:H| \leq k(G) \leq k(H) \cdot |G:H|$;
\item $k(H) \leq \sqrt{|G| k(G)}$; and
\item $k(G) \leq k(N) \cdot k(G/N)$.
\end{enumerate}
\end{lem}

\begin{proof}
Statements (1) and (3) can be found in \cite{gallagher} (see also \cite{Nagao}). Statement (2) follows from (1).
\end{proof}

In special cases we will need a straightforward consequence of the Clifford-Gallagher formula \cite[Page 18]{schmidbook}. The second statement of the following lemma follows from \cite[Proposition 8.5d]{schmidbook}.

\begin{lem}
\label{cliffordgallagher}
Let $\mathrm{Irr}(N)$ denote the set of complex irreducible characters of a normal subgroup $N$ of a finite group $G$. Then $S = G/N$ acts on $\mathrm{Irr}(N)$ in a natural way and let $I_{S}(\theta)$ denote the stabilizer of a character $\theta$ in $\mathrm{Irr}(N)$. Then we have $$k(G) \leq \sum_{\theta \in \mathrm{Irr}(N)} k(I_{S}(\theta))/|S:I_{S}(\theta)|.$$ Moreover if $N$ is a full direct power of a finite group $T$ and $S$ permutes the factors of $N$ transitively and faithfully, then $k(G) \leq k(T \wr S)$.
\end{lem} 

For a non-negative integer $n$ let the number of partitions of $n$ be denoted by $p(n)$. This is the number of conjugacy classes of the symmetric group $S_{n}$. In 1918 Hardy and Ramanujan \cite{HR} and independently but later Uspensky \cite{U} proved the following asymptotic formula. $$p(n) \sim \frac{e^{\pi \sqrt{2n/3}}}{4n \sqrt{3}}.$$ In 1937 Rademacher \cite{Rademacher} gave a series expansion of $p(n)$, however here we will only need the following lower and upper bounds.

\begin{lem}
\label{erdos}
Let $n \geq 1$ be an integer. Then $e^{2.5 \sqrt{n}}/13 n < p(n) < e^{\pi \sqrt{2n/3}}$.
\end{lem}

\begin{proof}
For the upper bound see \cite{erdos} and for the lower bound see \cite{Ma2}.
\end{proof}

\section{Primitive groups}
\label{sectionprimitive}

A transitive permutation group $G$ is called primitive if the stabilizer of any point is a maximal subgroup in $G$. This is equivalent to saying that the only blocks of imprimitivity for $G$ are the singleton sets and the whole set on which $G$ acts. The symmetric and alternating groups, $S_n$ and $A_n$, are examples of primitive permutation groups. In this section we will extend Corollary 2.15 (i) of \cite{LP} and Theorem 1.3 (i) of \cite{Ma3} to show Theorem \ref{subprim}. This result heavily depends on Theorem 1.1 of \cite{Ma1} and also on \cite{GAP}.

\begin{thm}
\label{subprim}
Let $G$ be a primitive permutation group of degree $n$ different from $A_n$ and $S_n$. Then we have $k(H) \leq p(n)$ for every subgroup $H$ of $G$.
\end{thm}
\begin{proof}
Let $G$ be a primitive permutation group of degree $n$. If $H \leq G$ are subgroups of $S_{m} \wr S_{r}$ in its product action on $n = \binom{m}{k}^{r}$ points where $m \geq 5$ and $S_{m}$ acts on $k$-subsets for some $k$ with $1 \leq k < n$, then $k(H) \leq 2^{mr-1}$ by Theorem 2 of \cite{LP}. But for $(k,r) \not= (1,1)$ we have $$2^{mr-1} < \frac{e^{2.5 \sqrt{\binom{m}{k}^r}}}{13 \binom{m}{k}^{r}} < p(\binom{m}{k}^{r}) = p(n),$$ where the second inequality follows from Lemma \ref{erdos}. Thus we may exclude these cases from the discussion.

By Theorem 1.1 of \cite{Ma1}, we then know that $|G| < n^{1+[\log_{2}(n)]}$ or $G$ is one of the Mathieu groups in their $4$-transitive action. 

Again by Lemma \ref{erdos}, we see that $|G| < n^{1+[\log_{2}(n)]} < p(n)$ for $n \geq 1500$. Furthermore, by using the exact values of $p(n)$ available in \cite{GAP}, $|G| < p(n)$ is true even for $n \geq 1133$.

If $120 \leq n < 1133$ then $p(n) < |G| < n^{1+[\log_{2}(n)]}$ holds only if $n=1024$ and $G = AGL(10,2)$, $n=512$ and $G = AGL(9,2)$, $n=256$ and $G = AGL(8,2)$, or $n=511$, $255$, $190$, $171$, $153$, $144$, $136$, $128$, $127$, $121$, or $120$ (this was also obtained by \cite{GAP}).

If $G$ is any of these exceptional cases (with $n \geq 120$) and is not a subgroup of $S_{m} \wr S_r$ in its product action discussed above, then $k(G)|G| < {p(n)}^{2}$, which forces $k(H) < p(n)$ for any subgroup $H$ of $G$ (by (2) of Lemma \ref{ineq}). Furthermore if $n \leq 119$ then we again have $k(G)|G| < {p(n)}^{2}$, unless $n = 64$ and $G = AGL(6,2)$, or $n \leq 32$ and $G$ is almost simple or of affine type. Both these statements were derived by \cite{GAP}.  

For almost simple primitive groups $G$ of degrees $n$ at most $32$ (including the $4$-transitive Mathieu groups but excluding $A_n$ and $S_n$) we can compute the subgroup lattice of $G$ by \cite{GAP} and so the claim can be checked for all subgroups $H$ of $G$. Thus we may assume that $G$ is an affine primitive permutation group of degree $64$ or at most $32$. 

We must show that if $H$ is a subgroup of $AGL(m,p)$ with $n= p^m \leq 64$, then $k(H) \leq p(n)$.  
If $m=1$ then it is easy to see that $k(H) \leq p = n \leq p(n)$. If $m=2$ and $p=5$ or $7$, or if $p^{m} = 27$, then $|AGL(2,p)| k(AGL(2,p)) < {p(n)}^{2}$ and we may apply (2) of Lemma \ref{ineq}. Thus we may assume that $p=2$ or $3$. The full subgroup lattice of $AGL(m,p)$ can be computed by \cite{GAP} for all remaining cases except $(m,p) = (5,2)$ and $(m,p) = (6,2)$, and thus the validity of the inequality $k(H) \leq p(n)$ can be checked directly.

Let $m=5$ and $p=2$. Any subgroup of $GL(5,2)$ has less than $260$ conjugacy classes (this can be obtained by \cite{GAP} by viewing $GL(5,2)$ as a permutation group on $31$ points), and so (3) of Lemma \ref{ineq} gives $k(H) < 260 \cdot 32 < p(32)$ for any subgroup $H$ of $AGL(5,2)$.
 
Let $m=6$ and $p=2$. Put $N = O_{2}(H)$. The factor group $H/N$ can be viewed as a completely reducible subgroup on a vector space of size $64$ (see \cite[Page 554]{LP}). We claim that $k(H/N) \leq 63$. For this observe that for irreducible linear subgroups $T$ of $GL(V)$ we have $k(T) < |V|$ whenever $V$ is a vector space of size a power of $2$ at most $64$. (This can be checked by \cite{GAP} by going through stabilizers of all affine primitive permutation groups of degrees a power of $2$ at most $64$.) Then, by using part (3) of Lemma \ref{ineq}, induction, and noting that a normal subgroup of a completely reducible linear group also acts completely reducibly on the same vector space (Clifford's theorem), we obtain the claim.    

Let $S$ be a Sylow $2$-subgroup of $AGL(6,2)$ containing $N$. Suppose that $|S:N| \geq 64$. Then (3) of Lemma \ref{ineq} gives $k(H) \leq |N| \cdot k(H/N) \leq 2^{15} \cdot 63 < 2^{21} < p(64)$. Now suppose that $|S:N| \leq 16$. Then $k(N) \leq |S:N| \cdot k(S) \leq 16 \cdot 1430$, by (1) of Lemma \ref{ineq}, and so $k(H) \leq k(N) \cdot k(H/N) \leq 16 \cdot 1430 \cdot 63 < p(64)$. So the only case missing is when $|S:N| = 32$. We would like to bound $k(N)$ in this case. Let $S_{1}$ be a maximal subgroup of $S$ containing $N$. By \cite{GAP} we know that $k(S_{1}) \leq 1723$ or $k(S_{1}) = 1768$. In the first case we have $k(N) \leq 16 \cdot 1723$, and so $k(H) \leq 16 \cdot 1723 \cdot 63 < p(64)$. So suppose that the second case holds. Then let $S_{2}$ be a maximal subgroup in $S_{1}$ containing $N$. By \cite{GAP} again, we know that $k(S_{2}) \leq 2240$, and so $k(N) \leq 8 \cdot 2240$. This gives $k(H) \leq 8 \cdot 2240 \cdot 63 < p(64)$.
\end{proof}

A straightforward consequence of Theorem \ref{subprim} is the following.

\begin{cor}
\label{subnormal}
If $H$ is a subnormal subgroup of a primitive permutation group of degree $n$, then $k(H) \leq p(n)$.
\end{cor}

\begin{proof}
If $H = S_n$ then this is clear. If $H = A_n$, then this follows from \cite[Lemma 2.3]{Ma3}. Otherwise apply Theorem \ref{subprim}.
\end{proof}

\section{Transitive groups}
\label{sectiontransitive}

In this section we will give an upper bound in terms of the partition function for $k(G)$ when $G$ is a transitive permutation group. This result depends on Theorem \ref{subprim} and is used in the proof of Theorem \ref{main}.

\begin{thm}
\label{part}
Let $G$ be a transitive permutation group of degree $n$ with point stabilizer $H$. Consider a chain $$H = H_0 < H_1 < \ldots < H_t = G$$with $H_i$ maximal in $H_{i+1}$ for $i=0,\ldots,t-1$ and call $a_i := |H_i:H_{i-1}|$ for $i=1,\ldots,t$, so that $a_1 \cdots a_t = |G:H| = n$. Then $$k(G) \leq {(p(a_1)^{1/a_1} p(a_2)^{1/a_1a_2} \cdots p(a_{t-1})^{1/a_1\cdots a_{t-1}} p(a_t)^{1/a_1 \cdots a_t})}^n.$$
\end{thm}

\begin{proof}
Let $G$ be a minimal counterexample to the statement of the theorem with a fixed chain of subgroups. By Corollary \ref{subnormal}, we may assume that $t \geq 2$. We now construct a subnormal filtration as in \cite{schmid}. Let $B_0$ be the core of $H_1$ in $G$, so that $G/B_0$ is a transitive permutation group of degree $n/a_1$. Let $N$ be the core of $H = H_0$ in $H_1$, so that $H_1/N$ is a primitive permutation group of degree $a_1$. Let $\{x_i\}_{1 \leq i \leq n/a_1}$ be a set of representatives for the right cosets of $H_1$ in $G$, with $x_1=1$, and define inductively $B_i := B_{i-1} \cap N^{x_i}$ for $i \geq 1$. Then $B_i = B_{i-1} \cap B_1^{x_i}$ and since $N$ is normal in $H_1$ and $H$ is core-free, $$B_{n/a_1} \subseteq \bigcap_{i=1}^{n/a_1} N^{x_i} = \bigcap_{g \in G} N^g \subseteq \bigcap_{g \in G} H^g = \{1\}.$$ We obtain a subnormal filtration (grading) $B = B_0 \rhd B_1 \rhd \cdots \rhd B_{n/a_1} = \{1\}$. Observe that $B_i \unlhd B_0$ for all $0 \leq i \leq n/a_1$, this is easily seen by induction on $i$: since $B_0 \unlhd G$ we have $B_1^{x_i} \unlhd B_0^{x_i} = B_0$ and hence $B_i = B_{i-1} \cap B_1^{x_i} \unlhd B_0$. Let $L := B_0 \cap N$. We have $$B_i/B_{i+1} = B_i/B_i \cap B_1^{x_{i+1}} = B_i/B_i \cap L^{x_{i+1}} \cong B_i L^{x_{i+1}}/L^{x_{i+1}} \unlhd B_0/L^{x_{i+1}} \cong B_0/L.$$ Since $B_0/L \cong B_0N/N \unlhd H_1/N$, each $B_i/B_{i+1}$ is isomorphic to a subnormal subgroup of the primitive group $H_1/N$ of degree $a_1$. By Corollary \ref{subnormal}, $k(B_i/B_{i+1}) \leq p(a_1)$ for all $i$. Now consider the chain $H_1/B < H_2/B < \ldots < H_{t-1}/B < H_t/B = G/B$. Each subgroup of the chain is maximal in the following one hence by minimality of $G$ the theorem holds for $G/B$ relative to this chain and hence
\begin{eqnarray}
k(G) & \leq & k(B) k(G/B) \leq \Big( \prod_{i=0}^{n/a_1-1} k(B_i/B_{i+1}) \Big) \cdot k(G/B) \nonumber \\ & \leq & p(a_1)^{n/a_1} \cdot (p(a_2)^{(n/a_1)/a_2} \cdots p(a_t)^{(n/a_1)/(a_2 \cdots a_t)}) \nonumber \\ & = & p(a_1)^{n/a_1} p(a_2)^{n/a_1a_2} \cdots p(a_{t-1})^{n/a_1 \cdots a_{t-1}} p(a_t). \nonumber
\end{eqnarray}
The proof is complete.
\end{proof}

\section{Proof of Theorem \ref{main}}
\label{sectionmain}

In this section we will prove our main result. The first lemma enables us to deal with cases when $n$ is relatively small. 

\begin{lem}
\label{l3}
If $G$ is a permutation group of degree $n$ all of whose orbits have lengths at most $23$ then $k(G) \leq 5^{n/4}$.
\end{lem}
\begin{proof}
By induction on $n$, as in Lemma 3.1 of \cite{Ma3}, we may assume that $G$ is transitive. For transitive groups the claim can be checked by \cite{GAP}. 
\end{proof}

By \cite{hulpke} all transitive permutation groups of degree at most $30$ are known therefore the $23$ in Lemma \ref{l3} could perhaps be replaced by $30$ (or even $31$) but it is not clear to what extent this possible improvement could be of help.

Now we proceed to the proof of Theorem \ref{main}. Many of the computations below have been performed by \cite{GAP}, but we will not point this out in all cases.

Let $G$ be as in the statement of the theorem. It acts faithfully on a set $\Omega$ of size $n$.

We proceed by induction on $n$. By Lemma \ref{l3} we can assume that $n \geq 24$. Suppose $G$ is intransitive and let $O$ be a nontrivial orbit of $G$ of size $1 < r < n$. Let $N$ be the kernel of the action of $G$ on $O$. Then $N$ acts faithfully on $n-r$ points and $G/N$ acts faithfully on $r$ points hence if $r,n-r \geq 4$ then $$k(G) \leq k(N) \cdot k(G/N) \leq 5^{(n-r-1)/3} \cdot 5^{(r-1)/3} < 5^{(n-1)/3}.$$ If $r \leq 3$ then $k(G/N) \leq r$, and if $n-r \leq 3$ then $k(N) \leq n-r$, from which the result follows likewise. Hence we may assume that $G$ is transitive.

Let $H$ be the stabilizer of $\alpha \in \Omega$ in $G$. If $H$ is maximal in $G$ then $G$ is a primitive permutation group and thus by Theorem \ref{part} and Lemma \ref{erdos} we have $k(G) \leq p(n) \leq e^{\pi \sqrt{2n/3}}$ and this is at most $5^{(n-1)/3}$ for $n \geq 25$.

So assume that $H$ is not maximal in $G$ and let $K$ be such that $H < K < G$. Let $a := |K:H|$ and $b := |G:K|$. Notice that the $K$-orbit $\Delta$ containing $\alpha$ is a block of imprimitivity for the action of $G$. Let $B$ be the kernel of the action of $G$ on the block system $\Sigma$ associated to $\Delta$, in other words, $B$ is the normal core of $K$ in $G$. $G/B$ is a transitive permutation group of degree $b$. By taking subsequent kernels on the blocks (i.e. arguing as in the proof of Theorem \ref{part}) we find a subnormal sequence $B_0=B \unrhd B_1 \unrhd \ldots \unrhd B_b = \{1\}$ such that each factor group $B_i/B_{i+1}$ can be considered as a permutation group of degree $a$.

If $a$ and $b$ are both at least $4$ then we may apply induction and find $$k(G) \leq k(B) \cdot k(G/B) \leq (5^{(a-1)/3})^{b} \cdot 5^{(b-1)/3} = 5^{(n-1)/3}.$$So we may assume that whenever $H < L < G$ either $|G:L| \leq 3$ or $|L:H| \leq 3$.

If both $a$ and $b$ are at most $3$ then $n \leq 9$ and the result follows from Lemma \ref{l3}. Assume that $4 \leq a \leq 23$ and $b \leq 3$. Then $k(G/B) \leq 3$ hence since the orbits of $B$ have all size at most $23$ by Lemma \ref{l3} we have $k(G) \leq k(B) k(G/B) \leq 5^{n/4} \cdot 3$ which is at most $5^{(n-1)/3}$ since $n \geq 24$.

We are in one of the following cases.

\begin{enumerate}
\item $H$ is maximal in $K$ and $b = |G:K| \in \{2,3\}$, $a \geq 24$ (consider the block system associated to $K$).
\item $K$ is maximal in $G$ and $a = |K:H| \in \{2,3\}$.
\item There exists a subgroup $L < G$ such that $H < K < L < G$ with $K$ maximal in $L$, $a=|K:H| \in \{2,3\}$, $c=|G:L| \in \{2,3\}$, and $q=|L:K| \geq 24/a$ (consider the block system associated to $L$).
\end{enumerate}

We consider the cases separately. In the following ``filtration argument'' refers to the argument used in the proof of Theorem \ref{part}. If $B \leq A$ are subgroups of $G$, by ``filtration associated to $A$ and $B$'' we mean the filtration of the kernel of the action of $A$ on the system of blocks associated to $B$ obtained as in the proof of Theorem \ref{part}.

Case 1. By Theorem \ref{part}, since $p(b) \leq b$ we have $k(G) \leq p(a)^b b$. Thus it is sufficient to show that $p(a)^b b \leq 5^{(ab-1)/3}$, i.e. $p(a) \leq ((5^{(ab-1)/3})/b)^{1/b}$. For this it is sufficient to show that $p(a) \leq ((5^{(2a-1)/3})/3)^{1/3}$ for $a \geq 24$. If $a \geq 55$ this follows from the bound $p(a) \leq e^{\pi \sqrt{2n/3}}$ (Lemma \ref{erdos}), and if $24 \leq a \leq 54$ it follows by inspection.

Case 2. In this case $G/B$ is a primitive group of degree $b$. Applying the filtration argument used in the proof of Theorem \ref{part}, since $p(a) \leq a$ we find $k(G) \leq a^b k(G/B)$ and it is enough to prove that $a^b k(G/B) \leq 5^{(ab-1)/3}$, i.e. (*) $k(G/B) \leq ((5^{(ab-1)/3})/a^b) = (5^{(a-1/b)/3}/a)^b$. Recall that $ab = n \geq 24$. If $a = 3$ then $b \geq 8$, now $p(b) \leq (5^{(3-1/8)/3}/3)^b$ follows from the bound $p(b) \leq e^{\pi \sqrt{2b/3}}$ (Lemma \ref{erdos}) if $b \geq 34$ and by inspection if $8 \leq b \leq 33$. Suppose now $a = 2$, so that $b \geq 12$. If $b=12$ let $S$ be a block stabilizer, then $|G:S|=b$ and $S$ is a permutation group on $24$ points having at least $2$ orbits hence by Lemma \ref{l3} we have $k(G) \leq 12 \cdot k(S) \leq 12 \cdot 5^{6}$ and this is less than $5^{23/3}$. Let $b \in \{13,14,15\}$. Then using the fact that any primitive group of degree $b$ different from $S_b$ has at most $k(A_b)$ conjugacy classes we see that (*) holds unless $G/B \cong S_b$. If $B$ is not elementary abelian of rank $b$ then the filtration argument implies $k(G) \leq a^{b-1} k(G/B) \leq 5^{(ab-1)/3}$. So assume that $B \cong C_2^b$ and $G/B \cong S_b$. Then by the Clifford-Gallagher formula (Lemma \ref{cliffordgallagher}) $k(G) \leq k(C_2 \wr S_b)$ which is at most $5^{(n-1)/3}$ by \cite{GAP}. If $16 \leq b \leq 55$ then (*) holds by inspection using $k(G/B) \leq p(b)$, and if $b \geq 56$ it follows from the bound $p(b) \leq e^{\pi \sqrt{2b/3}}$ (Lemma \ref{erdos}).

Case 3. By Theorem \ref{part}, since $p(a) \leq a$ and $p(c) \leq c$ we have $k(G) \leq a^b p(q)^c c$ where $b=qc$. We want to prove that $k(G) \leq 5^{(n-1)/3}$ where $n = ab = aqc$. If $a=3$ then it is sufficient to prove that $3^b p(q)^c c \leq 5^{(aqc-1)/3}$ for $q \geq 8$. Raising both sides to the power $1/c$ and rearranging, using the fact that $c^{1/c} \leq 1.5$ we see that it is sufficient to prove that $p(q) \leq \frac{1}{1.5} (5^{\frac{1}{3} (3-1/16)}/3)^q$ for $q \geq 8$. If $q \geq 31$ this follows from the bound $p(q) \leq e^{\pi \sqrt{2q/3}}$ (Lemma \ref{erdos}), and the case $8 \leq q \leq 30$ is checked by inspection.

Now assume that $a=2$ and $q \geq 16$. We prove that (**) $2^{cq} \cdot p(q)^c \cdot c \leq 5^{(2cq-1)/3}$. Raising both sides of (**) to the power $1/c$ and rearranging we see that it is enough to prove that $p(q) \leq \frac{1}{1.5} (5^{\frac{1}{3} (2-1/32)}/2)^q$, and for this it is enough to prove that $p(q) \leq \frac{1}{1.5} (1.43)^q$. If $q \geq 60$ this follows from the bound $p(q) \leq e^{\pi \sqrt{2q/3}}$ (Lemma \ref{erdos}), and if $16 \leq q \leq 59$ inequality (**) can be checked by inspection.

Now assume that $a=2$ and either $13 \leq q \leq 15$ or $(q,c)=(12,3)$. Every nontrivial subnormal subgroup of any primitive group of degree $q$ is a primitive group of degree $q$, a primitive group of degree $q$ which is not the full symmetric group $S_q$ has at most $k(A_q)$ conjugacy classes, and we have $k(A_{12})=43$, $k(A_{13})=55$, $k(A_{14})=72$, $k(A_{15})=94$. Moreover, the ratio $5^{(n-1)/3}/(2^{cq} \cdot p(q)^c \cdot c)$ is less than $2$. Thus we may assume that the kernel of the action of $G$ on the system of blocks associated to the primitive group $K/H_K$ is a direct product $C_2^{cq} = C_2^b$, indeed if this is not the case then using the filtration argument we see that $k(G) \leq 2^{cq-1} \cdot p(q)^c \cdot c \leq 5^{(n-1)/3}$. Consider the filtration $\mathcal{F}_1$ associated to $L$ and $K$. The two factors of this filtration are isomorphic to subnormal subgroups of the primitive group $L/K_L$ of degree $q$. Consider the filtration $\mathcal{F}_2$ associated to $L$ and $H$. By the Clifford-Gallagher formula (Lemma \ref{cliffordgallagher}) a fixed factor of $\mathcal{F}_2$ has at most $k(S_2 \wr A)$ conjugacy classes, where $A$ is a permutation group of degree $q$ isomorphic to a factor of $\mathcal{F}_1$. If no factor of $\mathcal{F}_1$ is isomorphic to $S_q$ then it is enough to show that $c \cdot k(A_q)^c \cdot 2^{cq} \leq 5^{(n-1)/3}$ which is true, and if there is a factor of $\mathcal{F}_1$ isomorphic to $S_q$ then since $k(S_2 \wr S_{13}) = 1770$, $k(S_2 \wr S_{14}) = 2665$ and $k(S_2 \wr S_{15}) = 3956$ by the Clifford-Gallagher formula (Lemma \ref{cliffordgallagher}) it is enough to show that $c \cdot k(S_2 \wr S_q) \cdot 2^{q(c-1)} \cdot p(q)^{c-1} \leq 5^{(n-1)/3}$ which is true.

Now assume that $(a,q,c) = (2,12,2)$. $K$ is the stabilizer of a block of size $2$ (there are $24$ such blocks). It acts on the $24$ points of a block system consisting of $12$ blocks of size $2$ intransitively, hence if $N$ denotes the kernel of this action we deduce $k(K/N) \leq 5^{24/4} = 5^6$. Now look at the (faithful) action of $N$ on the remaining $24$ points. If this action is intransitive then $k(N) \leq 5^{24/4}$ by Lemma \ref{l3}. If it is transitive then there is an induced transitive action of $N$ on the second block system of twelve blocks of size $2$. Since any transitive group of degree $12$ has at most $p(12)=77$ conjugacy classes (by \cite{GAP}), by Theorem \ref{part} we deduce $k(N) \leq 2^{12} \cdot 77$ and even $k(N) \leq 2^{11} \cdot 77$, in which case $k(G) \leq |G:K| \cdot k(K/N) \cdot k(N) \leq 24 \cdot 5^6 \cdot 2^{11} \cdot 77 \leq 5^{47/3}$, unless the kernel of the action of $N$ on the 12 blocks of size $2$ is a full direct product $C_2^{12}$. Suppose this is the case. Let $R$ be the kernel of the transitive action of $N$ on the twelve blocks of size $2$ of the second block system. If $k(N/R) \not \in \{65, 77\}$ then $k(N/R) \leq 55$ and $k(G) \leq |G:K| \cdot k(K/N) \cdot k(N) \leq 24 \cdot 5^6 \cdot 2^{12} \cdot 55 \leq 5^{47/3}$, so now assume $k(N/R) \in \{65,77\}$. It can be checked by \cite{GAP} that $k(S_2 \wr N/R) \in \{1165, 1265, 1960, 2210\}$. By the Clifford-Gallagher formula (Lemma \ref{cliffordgallagher}), $k(N) \leq k(S_2 \wr N/R) \leq 2210$ hence $k(G) \leq |G:K| \cdot k(K/N) \cdot k(N) \leq 24 \cdot 5^{6} \cdot 2210 \leq 5^{47/3}$.

\end{document}